\documentclass[twoside,11pt]{article}

\usepackage[notref,notcite]{showkeys}                    % show \label
\usepackage[colorlinks]{hyperref}                        % book-mark
\usepackage{amsmath,latexsym,amsfonts,indentfirst,amsthm,amsxtra,amssymb,bm}
\usepackage{mathrsfs,amssymb,dsfont}
\usepackage{txfonts}
\numberwithin{equation}{section}

 \newtheorem{lemma}{Lemma}[section]
 \newtheorem{proposition}[lemma]{Proposition}
 \newtheorem{theorem}{Theorem}

 \theoremstyle{remark}
 \newtheorem{remark}{Remark}[section]
\begin{document}

%--------------------------------------------------------------------------------------
%--------------------------------------------------------------------------------------
\title{\bf Nonlinear Stability of Large Amplitude Viscous Shock Wave
for General Viscous Gas}
\author{{\bf Lin He}\\[1mm]College of Mathematics, Sichuan University,
	\\ Chengdu, 610065,
	China.\\E-mail address: lin\_he@scu.edu.cn\\[2mm]{\bf Feimin Huang}\\[1mm] School of Mathematical Sciences, University of Chinese Academy of Sciences, \\Beijing 100049, China\\[1mm]
	and \\[1mm]Academy of Mathematics and Systems Science,\\ Chinese Academy of Sciences, Beijing 100190, China.\\
E-mail address: fhuang@amt.ac.cn}
\date{}

\maketitle

%--------------------------------------------------------------------------------------
%abstract------------------------------------------------------------------------------
\begin{abstract}
In the present paper, it is shown that the large amplitude viscous shock wave is nonlinearly stable for isentropic Navier-Stokes equations, in which the pressure could be general and includes $\gamma$-law, and the viscosity coefficient is a smooth function of density. The strength of shock wave could be arbitrarily large. The proof is given by introducing a new variable, which can formulate the original system into a new one,  and the elementary energy method introduced in \cite{Matsumura-Nishihara-1985-JJAM}.
\end{abstract}

%\tableofcontents

%---------------------------------------------------------------------------------------
\section{Introduction }
We consider the  1-D isentropic Navier-Stokes equations in Lagrangian coordinates, which reads  
\begin{equation}\label{1.1}
  \left\{ \begin{array} { l } { v _ { t } - u _ { x } = 0 }, \\
  { u _ { t } + p ( v ) _ { x } = \left( \frac { \mu ( v ) } { v } u _ { x } \right) _ { x } } , \end{array} \right.
\end{equation} 
with initial data 
\begin{equation}\label{initial}
(v , u) (x,0) = \left(v_ { 0 }, u_ { 0 } \right)(x) \rightarrow \left(v _ { \pm },u _ { \pm } \right) , ~ \text { as } ~ x \rightarrow \pm \infty,
\end{equation}
where $t>0, x\in\mathbb{R},$ and $v(x,t)=\frac{1}{\rho(x,t)}>0, u(x,t)$ denote respectively, the specific volume ($\rho$ is the density) and velocity, while $\mu(v)>0$ is the viscosity coefficient which is assumed to be a smooth function of $v$, and $p(v)>0$ is the smooth pressure satisfying
\begin{equation}\label{pressure}
  \begin{aligned}
    p' ( v )<0,\quad &p ''( v )\geq0,\quad p''(v)\nequiv0 ~\mbox{in any interval}.
  \end{aligned}
\end{equation}
Here the condition $p' ( v )<0$ is necessary to ensure that the corresponding inviscid part of \eqref{1.1} is strict hyperbolic and the $\gamma$-law pressure, i.e. $p(v)=v^{-\gamma}, \gamma\ge 1$, is included.  In view of the Chapman-Enskog expansion theory for rarefied gas dynamics (cf. \cite{Chapman-Cowling-1970,Grad-1963,Kawashima-Matsumura-Nishida-1979-CMP}), the viscosity is temperature-dependent, thus density-dependent in the isentropic flow \cite{LXY}, see also  \cite{Matsumura-Wang-2010-MAA}. 

The nonlinear stability of viscous shock wave is a fundamental problem in the theory of viscous conservation laws. Matsumura-Nishihara \cite{Matsumura-Nishihara-1985-JJAM} and  Goodman \cite{G} first studied the stability of viscous shock wave under a zero mass condition.  Since then, there have been  a large of literatures on this topic and various approaches and techniques like weighted characteristics, approximate Green functions and pointwise estiamtes are developed, see \cite{Liu-1985, Liu-1986, Liu-1997, Liu-Zeng-2009-CMP, Liu-Zeng-2015, K-M, S-X, HM, M-N-1994, Mascia-Zumbrun-2004-CPAM, Matsumura-Wang-2010-MAA, Vasseur-Yao-2016-CMS, Z} and references therein. It is noted that most of above works require the strength of shock wave is suitably small, that is, the shock is weak. From the physical piont of view, the strong shock (its strength is not necessarily small) is more meaningful. Along this direction, an interesting $L^1$ stability theory of large shock wave for scalar viscous conservation law was established in \cite{FS}. 
For $\gamma$-law pressure system, Matsumura-Nishihara \cite{Matsumura-Nishihara-1985-JJAM} showed that the viscous shock wave is stable if $(\gamma-1)|v_+-v_-|$ is small, that is,  when $\gamma\rightarrow1$, the strength of shock wave could  be large. This condition is later relaxed in \cite{K-M} to the condition that  $(\gamma-1)^2|v_+-v_-|$ is small, and finally removed in \cite{Matsumura-Wang-2010-MAA} and \cite{Vasseur-Yao-2016-CMS} when $\mu(v)=v^{-\alpha}, \alpha\ge 0$. Through different approach, Zumbrun et.al \cite{HM}, \cite{Mascia-Zumbrun-2004-CPAM}, \cite{Z} proved  the stability of large amplitude shock not only in 1-D but also in M-D.  

Motivated by \cite{Matsumura-Wang-2010-MAA} and \cite{Vasseur-Yao-2016-CMS}, we begin to investigate the stability of large amplitude shock wave of the system \eqref{1.1} for general pressure  $p(v)$ and viscosity $\mu{(v)}$, in which the pressure satisfies \eqref{pressure} and its second derivative $p''(v)$ could be zero somewhere, that is, the pressure is not necessary to be strictly convex. To achive the main purpose, we first establish the existence of large amplitude traveling wave of the system \eqref{1.1} by the method developed in \cite{Caflisch-1979-CPAM}. 
%The precise statement of the existence of traveling wave is formulated in Theorem 1 of section 2 below.  
Then motivated by \cite{Bresch-Desjardins-2003-CMP,Bresch-Desjardins-2006-JMPA}, \cite{Shelukhin-1984-JAMM} and \cite{Vasseur-Yao-2016-CMS}, we introduce a new variable $h(x,t)$ and translate the original system \eqref{1.1} into a new system \eqref{1.9} whose first equation is parabolic, while the seond one is hyperbolic. This new system is specially designed for the large amplitude shock wave so that a bad term  is cancelled in the elementary energy estimates, see \eqref{3.6} below. Thus the large amplitue shock wave is shown nonlinearly stable.
% and the precise statement is formulated in Theorem 2 of section 3 below.

We now formulate the main works of the present paper. 
%Because the pressure function $p(v)$ is not strictly convex, we need to reprove the existence of the viscous shock wave to system \eqref{1.1} firstly.
%In this section, we state the arguments on the existence of the viscous shock wave.
%We are concerned with the Cauchy problem of the system \eqref{1.1} with the following initial data:
%\begin{equation}\label{1.2}
%( v , u ) ( x , 0 ) = \left( v _ { 0 } , u _ { 0 } \right) ( x ) , \quad x \in \mathbb{R},
%\end{equation}
%and far field condition:
%\begin{equation}\label{1.3}
%( v , u ) ( x , t ) \rightarrow \left( v _ { \pm } , u _ { \pm } \right) , \quad \text { as } \quad x \rightarrow \pm \infty,t\geq0.
%\end{equation}
We are looking for a traveling wave $( \tilde{v} , \tilde{u} ) ( x-st )$ satisfying $(\tilde{v} , \tilde{u}) (\xi) \rightarrow \left(v _ { \pm } ,u _ { \pm }\right) , \text { as }~  \xi \rightarrow \pm \infty$, denoted by  viscous shock wave solution,  to the system \eqref{1.1}, where $\xi = x-st$.  Then we obtain the following ordinary differential equations (ODE):
\begin{equation}\label{1.4}
\left\{ \begin{array} { l } { - s \tilde { v } _ { \xi } - \tilde { u } _ { \xi } = 0 },\\
{-s \tilde{u}_{\xi}+p(\tilde{v})_{\xi}=\left(\frac{\mu(\tilde{v})\tilde{u}_{\xi}} {\tilde{v}}\right)_{\xi}}, \\
{ ( \tilde { v } , \tilde { u } ) ( \pm \infty ) = \left( v _ { \pm } , u _ { \pm } \right) }. \end{array} \right.
\end{equation}
Integrating the system \eqref{1.4} with respect to $\xi$ over $(\pm\infty,\xi)$ and using the fact that
$\tilde{u}_{\xi} \rightarrow 0 $ as $\xi \rightarrow \pm \infty $, we get
\begin{equation}\label{1.5}
\left\{ \begin{array} { l } { s \tilde { v } + \tilde { u } = s v _ { \pm } + u _ { \pm } },\\
{ - s \tilde { u } + p ( \tilde { v } ) - \frac { \mu(\tilde { v })\tilde { u } _ { \xi } } { \tilde { v } } = - s u _ { \pm } + p \left( v _ { \pm } \right) }, \end{array} \right.
\end{equation}
and this implies the Rankine-Hugoniot jump condition for shock waves
\begin{equation}\label{1.6}
\left\{ \begin{array} { l } { - s \left( v _ { + } - v _ { - } \right) = \left( u _ { + } - u _ { - } \right) }, \\
{ - s \left( u _ { + } - u _ { - } \right) = - \left( p \left( v _ { + } \right) - p \left( v _ { - } \right) \right) }, \end{array} \right.
\end{equation}
from which we can deduce that
\begin{equation}\label{1.7}
s = \pm \sqrt { - \frac { p \left( v _ { + } \right) - p \left( v _ { - } \right) } { v _ { + } - v _ { - } } }.
\end{equation}
We will only deal with the case that $s<0$ (while the case $s>0$ is similar), and the so-called entropy condition is 
\begin{equation}\label{1.8}
v _ { - } > v _ { + } , \quad u _ { - } > u _ { + }.
\end{equation}
From \eqref{1.5}, we have that
\begin{equation}\label{1.13}
\tilde { v } _ { \xi } = \frac { \tilde { v }  } {- s\mu(\tilde { v }) } \left[ s ^ { 2 } \left(  \tilde { v }- v _ { - } \right) + p ( \tilde { v } ) - p \left( v _ { - } \right) \right] =: B(\tilde{v}).
\end{equation}
Then we obtain the existence of the traveling wave
\begin{theorem}[Existence of the viscous shock wave]\label{existence} Under the condition \eqref{pressure}, there exists a unique smooth solution $(\tilde{v},\tilde{u})(\xi)$ of the system \eqref{1.4} up to a shift. Moreover, $u'(\xi)\le 0$.
	
\end{theorem}

Now we are ready to state our main theorem. Motivated by  \cite{Bresch-Desjardins-2003-CMP,Bresch-Desjardins-2006-JMPA}, \cite{Shelukhin-1984-JAMM} and \cite{Vasseur-Yao-2016-CMS}, we introduce the effective velocity $h = u - g(v)_x$, where
\begin{equation}\label{1.8-1}
g(z)=\int^{z}\frac{\mu(\eta)}{\eta}d\eta.
\end{equation}
Using the effective velocity, we can easily obtain the derivative estimates of density $\rho$. Indeed, Shelukhin \cite{Shelukhin-1984-JAMM} and Bresch-Desjardins \cite{Bresch-Desjardins-2003-CMP,Bresch-Desjardins-2006-JMPA} applied the new effective velocity $h$ to obtain the entropy estimates, see also  \cite{Vasseur-Yao-2016-CMS} for the stability of large shock wave of $\gamma$-law pressure. 

From \eqref{1.8-1}, it holds that 
\begin{equation}\label{1.9}
\left\{ \begin{array} { l } { v _ { t } - h _ { x } = \left( g'(v)v _ { x } \right) _ { x } }, \\
{ h _ { t } + p ( v ) _ { x } = 0 }, \end{array} \right.
\end{equation}
that is, the first equation of \eqref{1.9} is parabolic, while the second one becomes hyperbolic. Then the initial data \eqref{initial} is changed into 
\begin{equation}\label{1.10}
( v , h ) ( x , 0 ) = \left(v _ { 0 } ,h _ { 0 } \right) ( x ):=\left(v_{ 0 } ,u_ {0}-g'(v_0)v'_0\right)(x), 
\end{equation}
and the corresponding far field condition becomes
\begin{equation}\label{1.11}
( v , h ) ( x , t ) \rightarrow \left( v _ { \pm } , u _ { \pm } \right), ~ \text { as } ~x \rightarrow \pm \infty, ~t \geq 0.
\end{equation}

Let $\tilde{h}=\tilde { u } - g(\tilde{v})_x$,  then \eqref{1.4} is  equivalent to
\begin{equation}\label{1.12}
\left\{
\begin{array} { l }
{ -s\tilde{v}_{\xi}-\tilde{h}_{\xi}=\left( g'(\tilde{v})\tilde{v}_{\xi} \right)_{\xi} }, \\
{ -s\tilde{h}_{\xi}+p(\tilde{v})_{\xi}=0 } ,\\
{ (\tilde{v},\tilde{h}) (\pm\infty)=\left(v_{\pm},u_{\pm}\right) }.
\end{array}
\right.
\end{equation} 
Assume that
\begin{equation}\label{1.14}
(v_{0}-\tilde{v},h_{0}-\tilde{h}) \in H^{2} \cap L^{1} , \quad \inf_{x\in R^{1}}v_{0}(x)>0,  
\end{equation}
and
\begin{equation}\label{1.14-1}
\int\left(v_{0}-\tilde{v}\right) (x)dx=0, \quad \int\left(h_{0}-\tilde{h}\right) (x)dx=0.
\end{equation}
Based on \eqref{1.14-1}, we can define 
\begin{equation}\label{1.15}
\Phi_{0}(x)=\int_{-\infty}^{x}\left[v_{0}(z)-\tilde{v}(z)\right] dz, \quad \Psi_{0} (x)=\int_{-\infty}^{x} \left[h_{0}(z)-\tilde{h}(z)\right] dz.
\end{equation}
We further assume that
\begin{equation}\label{1.16}
\left(\Phi_{0},\Psi_{0}\right)(x)\in L^{2}.
\end{equation}
It is worthy to point out that the condition \eqref{1.14-1} can be relaxed to that there exists a shift $\alpha$ such that  
\begin{equation}\label{1.14-2}
\int\left(v_{0}(x)-\tilde{v}(x+\alpha)\right)dx=0, \quad \int\left(h_{0}(x)-\tilde{h}(x+\alpha)\right)dx=0.
\end{equation}
The conditon \eqref{1.14-2} is usually called by zero mass condition. The condition \eqref{1.14-1} corresponds to the case $\alpha=0$ in \eqref{1.14-2}. The main result of the present paper is 
\begin{theorem}[Main Theorem]\label{th1}
	Under the assumptions \eqref{1.14}, \eqref{1.14-1} and \eqref{1.16}, there exists a positive constant $\delta_0$ such that if $\|(\Phi_0,\Psi_0)\|_3\leq\delta_0$, then the Cauchy problem \eqref{1.1} has a unique global solution $(v,u)(x,t)$ satisfying
	\begin{equation}\label{1.18}
	\begin{array} { l }
	{ ( v - \tilde { v } , u - \tilde { u } ) \in C ^ { 0 } \left( [ 0 , + \infty ) ; H ^ { 2 } \right) , ~
		v - \tilde { v } \in L ^ { 2 } \left( [ 0 , + \infty ) ; H ^ { 3 } \right) } \\
	{ u - \tilde { u } \in L ^ { 2 } \left( [ 0 , + \infty ) ; H ^ { 2 } \right) },
	\end{array}
	\end{equation}
	and
	\begin{equation}\label{1.19}
	\sup _ { x \in R } | v ( x , t ) - \tilde { v } ( x - s t ) | \rightarrow 0 , \quad
	\sup _ { x \in R } | u ( x , t ) - \tilde { u } ( x - s t ) | \rightarrow 0 , ~ \text { as } t \rightarrow \infty.
	\end{equation}
\end{theorem}
\begin{remark}
	The theorem \ref{th1} means that any large amplitude shock wave is nonlinearly stable for compressible Navier-Stokes equations even for general pressure $p(v)$ and viscosity $\mu(v)$ under zero mass condition. We believe that the large amplitude shock wave is still stable for general small perturbation, without zero mass conditon, however, this will be left for future.
\end{remark}

\begin{remark}
	A clever weighted energy method was introduced in \cite{M-N-1994} for the non-convex nonlinearity. In the theorem \ref{th1}, the weighted estimates are not needed even the pressure $p$ could not be strictly convex. 
\end{remark}

The rest of this paper will be organized as follows. In section 2, we establish the existence of the large amplitude viscous shock wave.  In section 3, we reformulate the problem by the anti-derivatives of the perturbations around the viscous shock wave. In section 4, the a priori estimates are  established and in section 5,  the proof of the main theorem is completed.

\

\textbf{Notations:} Throughout this paper, the generic positive constants will be denoted by C without confusion, which may changes from line to line. Notations $L^p, H^s$ denote the usual Lebesgue space and Sobolev space on $\mathbb{R}$ with norms $\|\cdot\|_{L^p}$ and $\|\cdot\|_{s}$, respectively. For simplicity, $\|\cdot\|:=\|\cdot\|_{L^2}$.

\section{Existence of viscous shock wave}

The existence theorem \ref{existence} is based on the following  useful lemma.
\begin{lemma}[cf. Caflisch \cite{Caflisch-1979-CPAM}]\label{caflisch}
	For the ordinary differential equation \eqref{1.13}, if $B(\tilde{v})$ satisfies the following conditions
	\begin{itemize}
		\item[(i)] $B(\tilde{v})=0$ has only two roots $v_-,v_+$;
		\item[(ii)] $\left.\frac{d}{d\tilde{v}}B(\tilde{v})\right|_{\tilde{v}=v_-}>0,\left.\frac{d}{d\tilde{v}}B(\tilde{v})\right|_{\tilde{v}=v_+}<0$;
		\item[(iii)] $B(\tilde{v})<0$ when $v_+<\tilde{v}<v_-$.
	\end{itemize}
	Then there exists a unique global smooth solution $\tilde{v}(\xi)$ of  \eqref{1.13}, up to a shift, satisfying $\tilde{v}(\pm\infty)=v_{\pm}$.
\end{lemma}

\begin{proof}[Proof of theorem \ref{existence}]
To prove the existence theorem \ref{existence},	we only need to verify the above three conditions for the term $B(\tilde{v})$. 
	
\

	\text{\bf (ii)} By the entroy condition \eqref{1.8}, its straightforward to check that
	\begin{equation}\label{e1}
	\frac{d}{d\tilde{v}}B(\tilde{v})\left.\right|_{\tilde{v}=v_+}=\frac{v_+}{-s\mu(v_+)}f'(v_+)=\frac{v_+}{-s\mu(v_+)}(s^2+p'(v_+))<0;\quad
	\frac{d}{d\tilde{v}}B(\tilde{v})\left.\right|_{\tilde{v}=v_-}=\frac{v_-}{-s\mu(v_-)}f'(v_-)>0.
	\end{equation}
	%On one hand, from $f(v_-)=f(v_+)=0$, by mean value theorem, we have there $\exists v'\in(v_+,v_-)$\,s.t.\, $f'(v')=0$. On the other hand, by the monotonicity of
%	$f'(\tilde{v})$ we have $f'(v_+)\geq0,f'(v_-)\leq0$. If $f'(v_+)=0$, then $\forall\tilde{v}\in[v_+,v'],\,f'(\tilde{v})=0$, moreover we have $f(v')=0$ which is contradict to (i). Thus we have
%	\begin{equation*}
%	f'(v_+)>0;\quad f'(v_-)<0.
%	\end{equation*}
%	combine with $s<0$, we can prove (ii) immediately.
\

	\text{\bf (iii)} Since $B(v_+)=B(v_-)=0$, (iii) is a direct conclusion of (i).
	
\

It remains to verify \text{\bf (i)}. In view of \eqref{1.13}, it turns  to verify \text{\bf (i)} for $f(\tilde{v}):= s ^ { 2 } \left(  \tilde { v }- v _ { - } \right) + p ( \tilde { v } ) - p \left( v _ { - } \right)$ due to the fact that  $f(v_-)=f(v_+)=0$, $\frac { \tilde { v }  } {- s\mu(\tilde { v }) }>0$, R-H condition \eqref{1.6} and $\mu(\tilde{v})>0$. Assume that \text{\bf (i)} is not correct, then  there exists $v^{*}\in(v_+,v_-)$ such that $\, f(v^{*})=0$. 
Since $f(v_+)=f(v^{*})=f(v_-)=0$, thanks to the mean value theorem, there exist two points $v_1\in(v_+,v^{*})$ and $v_2\in(v^{*},v_-)$ such that $f'(v_1)=f'(v_2)=0$. Note that  $f''(\tilde{v})=p''(\tilde{v})\geq0$, $f'(\tilde{v})$ is monotone increasing and $f'(\tilde{v})=0$ for all $\tilde{v}\in[v_1,v_2]$.  Since $v^*\in (v_1,v_2)$, $f(\tilde{v})=0$ holds for all $\tilde{v}\in[v_1,v_2]$ due to $f(v^*)=0$. In particular, $f(v^{*})=f'(v^{*})=0$. 
	
Set $\hat{v}=:\inf\{\tilde{v}\in (v_+,v_1]|f(\tilde{v})=f'(\tilde{v})=0\}.$	Note that $f'(v_+)<0$, then $\hat{v}>v_+$. Repeating the previous argument for $v_1$, we know that there exists $v_*\in (v_+,\hat{v})$ such that 
$f(v_*)=f'(v_*)=0$ which is a contradiction with 
the fact that $\hat{v}$ is the minimum of the points $\tilde{v}$ satisfying $f(\tilde{v})=f'(\tilde{v})=0$. Thus  \text{\bf (i)} is verified for $B(\tilde{v})$.
	
Finally,  we conclude from \text{\bf (iii)} that $u'(\xi)=-sv'(\xi)=-sB(\tilde{v})\le 0$. 
\end{proof}

%\section{Main Theorem}

%\begin{theorem}\label{th2}
%  Under the assumptions \eqref{1.14} and \eqref{1.16}, there exists a positive constant $\delta_0$ such that if $\|(\Phi_0,\Psi_0)\|_2\leq\delta_0$, then the Cauchy problem \eqref{1.9}$-$\eqref{1.11} has a unique global solution $(v,h)(x,t)$ satisfying
%  \begin{equation}
%    \begin{array}{l}
%      {( v - \tilde { v } , h - \tilde { h } ) \in C ^ { 0 } \left( [ 0 , + \infty ) ; H ^ { 1 } \right) , \quad
 %     v - \tilde { v } \in L ^ { 2 } \left( [ 0 , + \infty ) ; H ^ { 2 } \right),}\\
%      {h - \tilde { h } \in L ^ { 2 } \left( [ 0 , + \infty ) ; H ^ { 1 } \right)}
%    \end{array}
%  \end{equation}
 % and
 % \begin{equation}
%    \sup _ { x \in R ^ { 1 } } | ( v , h ) ( x , t ) - ( \tilde { v } , \tilde { h } ) ( x - s t ) | \rightarrow 0 , \quad \text { as } t \rightarrow \infty.
%  \end{equation}
%\end{theorem}
\section{Reformulation of the problem}
By changing the variables $(x,t)\rightarrow(\xi=x-st,t)$,   the systems \eqref{1.9} and \eqref{1.12} beome, respectively, 
\begin{equation}\label{2.1}
  \left\{
  \begin{array} { l }
  {v_{t}-sv_{\xi}-h_{\xi}=\left(g'(v)v_{\xi}\right)_{\xi} }, \\
   {h_{t}-sh_{\xi}+p(v)_{\xi}=0},
   \end{array}
   \right.
\end{equation}
and
\begin{equation}\label{2.1-1}
\left\{
\begin{array} { l }
{\tilde{v}_{t}-s\tilde{v}_{\xi}-\tilde{h}_{\xi}=\left(g'(\tilde{v})\tilde{v}_{\xi}\right)_{\xi} }, \\
{\tilde{h}_{t}-s\tilde{h}_{\xi}+p(\tilde{v})_{\xi}=0}.
\end{array}
\right.
\end{equation}
Then we define 
\begin{equation}\label{2.2}
  \Phi ( \xi , t ) = \int _ { - \infty } ^ { \xi } [ v ( z , t ) - \tilde { v } ( z ) ] d z , \quad
  \Psi ( \xi , t ) = \int _ { - \infty } ^ { \xi } [ h ( z , t ) - \tilde { h } ( z ) ] d z.
\end{equation}
Substitute \eqref{2.1-1} from \eqref{2.1} and integrate the resulting system with respect to $\xi$, we have
\begin{equation}\label{2.3}
  \left\{
  \begin{array} { l }
  {\Phi_{t}-s\Phi_{\xi}-\Psi_{\xi}-g'(\tilde{v})\Phi_{\xi\xi}-g''(\tilde{v})\tilde{v}_{\xi}\Phi_{\xi}=G }, \\
  {\Psi_{t}-s\Psi_{\xi}+p'(\tilde{v})\Phi_{\xi}=-p(v|\tilde{v}) },
  \end{array}
  \right.
\end{equation}
where
\begin{equation}\label{2.6}
\begin{array}{l}
{G=\left( g'(v) - g'(\tilde{v}) \right) \Phi_ { \xi \xi } + \left( g'(v) - g'(\tilde{v}) -g''(\tilde{v})\Phi_{\xi}\right) \tilde { v } _ { \xi }},\\
{p ( v | \tilde { v } )= p ( v ) - p ( \tilde { v } ) - p ^ { \prime } ( \tilde { v } ) ( v - \tilde { v } ) }.
\end{array}
\end{equation}
From \eqref{1.15}, the initial data belongs to 
\begin{equation}\label{2.4}
  ( \Phi, \Psi) ( \xi , 0 ) = \left( \Phi_ { 0 } , \Psi_ { 0 } \right) ( \xi ) \in H ^ { 3}.
\end{equation}
%\begin{equation}\label{2.5}
%  \Phi_ { 0 } ( \xi ) = \int _ { - \infty } ^ { \xi } \left[ \Phi_ { 0 } ( z ) - \tilde { v } ( z ) \right] d z , \quad \Psi_ { 0 } ( \xi ) = \int _ { - \infty } ^ { \xi } \left[ \Psi_ { 0 } ( z ) - \tilde { h } ( z ) \right] d z,
%\end{equation}
%and

We will seek the solutions in the functional space $X(0,T)$ for any $0\leq T<\infty$,
\begin{equation}\label{f}
  \begin{aligned} X_\delta(0 ,T) = \{ ( \Phi, \Psi) & \in C \left( 0 , T ; H ^ { 3 } \right) \left| \Phi_ { \xi } \in L ^ { 2 } \left( 0 , T ; H ^ { 3 } \right) , \Psi_ { \xi } \in L ^ { 2 } \left( 0 , T ; H ^ { 2 } \right) \right. \\
  &\sup _ { 0 \leq t \leq T } \| ( \Phi, \Psi) ( t ) \| _ { 3}  \leq \delta \}. \end{aligned}
\end{equation}
where the constant $\delta<<1$ is small. Since the local existence of the solution can be shown by standard way as contraction mapping principle, to prove the global existence of solutions of \eqref{2.3}, we only need to establish the following a priori estimates.

\begin{proposition}[A priori estimate]\label{prop2.1}
  Suppose that  $(\Phi,\Psi) \in X_\delta(0,T)$ is the solution of systems \eqref{2.3}-\eqref{2.6} for some time $T>0$, there exists a positive constant $\delta_0$ independent of $T$ such that if \begin{equation}\sup\limits_{t\in[0,T]} \|(\Phi,\Psi)(t)\|_{3}\leq\delta\le\delta_0,
  \end{equation} 
  then for any $t\in[0,T]$, it holds that
  \begin{equation}\label{2.7}
    \| ( \Phi, \Psi) ( t ) \| _ { 3 } ^ { 2 } + \int _ { 0 } ^ { t } \left( \left\| \Phi_ { \xi } ( \tau ) \right\| _ { 3 } ^ { 2 } + \left\| \Psi_ { \xi } ( \tau ) \right\| _ { 2 } ^ { 2 } \right) d \tau \leq C_0 \left\| \left( \Phi_ { 0 } , \Psi_ { 0 } \right) \right\| _ { 3 } ^ { 2 },
  \end{equation}
  where $C_0>1$ is a constant independent of T.
\end{proposition}

Once Proposition \ref{prop2.1} is proved, we can extend the unique local solution $(\Phi,\Psi)$ to $T=+\infty$ by the standard continuation process. Thus we have the global solution as below.
\begin{theorem}\label{th3}
  Suppose that $\left(\Phi_{0},\Psi_{0}\right) \in H^{3}$, then there exists a positive constant $\delta_{1}\le \frac{\delta_0}{\sqrt{C_0}}$ such that if $\left\|\Phi_{0},\Psi_{0}\right\|_{3} \leq \delta_{1}$, then the Cauchy problem \eqref{2.3}-\eqref{2.6} has a unique global solution $(\Phi,\Psi)\in X_{\delta_0}(0,\infty)$ satisfying
  \begin{equation}
    \sup_{t\geq0}\|(\Phi,\Psi)(t)\|_{3}^{2}+\int_{0}^{\infty}\left(\left\|\Phi_{\xi}(t)\right\|_{3}^{2}+\left\|\Psi_{\xi}(t)\right\|_{2}^{2}\right) dt
    \leq C_0 \left\| \left(\Phi_{0},\Psi_{0}\right) \right\|_{3}^{2}.
  \end{equation}
 %  and the long time behavior
%  \begin{equation}
 %   \sup_{\xi\in R^{1}} \left| (\Phi,\Psi)_{\xi} (\xi,t) \right| \rightarrow 0 , \quad \text { as } t \rightarrow \infty.
%  \end{equation}
\end{theorem}

\section{A priori estimate}
Assume that the systems \eqref{2.3}-\eqref{2.6} has a solution $( \Phi, \Psi) \in X_\delta(0,T), \delta<<1$ for some $T>0$, that is, 
\begin{equation}\label{3.1}
  \sup _ { t \in [ 0 , T ] } \| ( \Phi, \Psi) ( t ) \| _ { 2 } \leq \delta,
\end{equation}
then it follows from the Sobolev inequality that  $\frac{1}{2}v_+\le v\le \frac{3}{2}v_-$ and  
\begin{equation}\label{3.2}
  \sup _ { t \in [ 0 , T ] } \left\{ \| ( \Phi, \Psi) ( t ) \| _ { L^ \infty } + \| ( v - \tilde { v } , h - \tilde { h } ) ( t ) \| _ { L ^ { \infty } } \right\} \leq \delta.
\end{equation}

\begin{lemma}\label{le3.1}
  Under the a priori assumption \eqref{3.1}, we have
  \begin{equation}\label{3.3}
    \begin{array} { l }
    { p ( v | \tilde { v } ) \leq C \Phi_ { \xi } ^ { 2 } }, \\
    { \left| p(v|\tilde{v})_{\xi}\right|\leq C\left(|\Phi_{\xi\xi} | |\Phi_{\xi}| + |\tilde{v}_\xi |\Phi_{\xi}^{2}\right)}, \\
    { \left| p(v|\tilde{v})_{\xi\xi}\right|\leq C\left(|\Phi_{\xi\xi\xi}| |\Phi_{\xi}|+\Phi_{\xi}^{2}+\Phi_{\xi\xi}^{2}+|\Phi_{\xi}||\Phi_{\xi\xi}|^2 \right) },
    \end{array}
  \end{equation}
  and
  \begin{equation}\label{3.4}
    \begin{array} { l }
    { | G | \leq C \left( | \Phi_ { \xi }| | \Phi_ { \xi \xi } | + | \tilde { v } _ { \xi }| | \Phi_ { \xi } |^2 \right) }, \\
    { \left|G_{\xi}\right|\leq C\left(\Phi_{\xi\xi}^{2}+|\Phi_{\xi}| |\Phi_{\xi\xi}|^2+|\Phi_{\xi}| |\Phi_{\xi\xi\xi}|+\left|\Phi_{\xi}\right|^2\right) }.
    \end{array}
  \end{equation}
\end{lemma}

\begin{proof}
  From \eqref{2.6}, a direct calculation gives that
  \begin{equation*}
    \begin{aligned}
      p ( v | \tilde { v } )_{\xi}=&\left(p' ( v ) - p' ( \tilde { v } )\right)\Phi_{\xi\xi}
      +\left(p' ( v ) - p' ( \tilde { v } )-p'' ( \tilde { v } )\Phi_{\xi}\right)\tilde { v }_{\xi},\\
      p ( v | \tilde { v } )_{\xi\xi}=&\left(p' ( v ) - p' ( \tilde { v } )\right)\Phi_{\xi\xi\xi}
      +\left(p' ( v ) - p' ( \tilde { v } )-p'' ( \tilde { v } )\Phi_{\xi}\right)\tilde { v }_{\xi\xi}\\
      &+\left[\left(p'' ( v ) - p'' ( \tilde { v } )\right)\Phi_{\xi\xi}+\left(p'' ( v ) - p'' ( \tilde { v } )\right)\tilde{v}_{\xi}
      +p''(\tilde{v})\Phi_{\xi\xi}\right]\Phi_{\xi\xi}\\
      &+\left[\left(p'' ( v ) - p'' ( \tilde { v } )\right)\Phi_{\xi\xi}+\left(p'' ( v ) - p'' ( \tilde { v } )-p'''(\tilde{v})\Phi_{\xi}\right)\tilde{v}_{\xi}
      \right]\tilde{v}_{\xi},
    \end{aligned}
  \end{equation*}
  and
  \begin{equation*}
    \begin{aligned}
      G=&\left(g'(v)-g'(\tilde{v})\right)\Phi_{\xi\xi}+\left(g'(v)-g'(\tilde{v})-g''(\tilde{v})\Phi_{\xi}\right)\tilde{v}_{\xi},\\
      G_{\xi}=&\left(g'(v)-g'(\tilde{v})\right)\Phi_{\xi\xi\xi}+\left(g'(v)-g'(\tilde{v})-g''(\tilde{v})\Phi_{\xi}\right)\tilde{v}_{\xi\xi}\\
      &+\left[\left(g'' ( v ) - g'' ( \tilde { v } )\right)\Phi_{\xi\xi}+\left(g'' ( v ) - g'' ( \tilde { v } )\right)\tilde{v}_{\xi}
      +g''(\tilde{v})\Phi_{\xi\xi}\right]\Phi_{\xi\xi}\\
      &+\left[\left(g'' ( v ) - g'' ( \tilde { v } )\right)\Phi_{\xi\xi}+\left(g'' ( v ) - g'' ( \tilde { v } )-g'''(\tilde{v})\Phi_{\xi}\right)\tilde{v}_{\xi}
      \right]\tilde{v}_{\xi}.
    \end{aligned}
  \end{equation*}
  By virtue of \eqref{3.2}, the boundedness of $|\tilde{v}_{\xi}|,|\tilde{v}_{\xi\xi}|$, it is straightforward to imply \eqref{3.3} and \eqref{3.4}.
\end{proof}

With the preparations above, we are ready to show the proposition \ref{prop2.1}. First we have the following basic energy estimate.
\begin{lemma}\label{le3.2}
  Under the assumptions of proposition \ref{prop2.1}, it holds that 
  \begin{equation}\label{3.5}
    \begin{aligned}
    & \int_{-\infty}^\infty \left( \Phi^ { 2 } + \Psi^ { 2 } \right) d \xi - s \int _ { 0 } ^ { t } \int_{-\infty}^\infty \left( \frac { 1 } { p ^ { \prime } ( \tilde { v } ) } \right) _ { \xi } \Psi^ { 2 } d \xi d \tau + \int _ { 0 } ^ { t } \int_{-\infty}^\infty \Phi_ { \xi } ^ { 2 } d x d \tau \\
    \leq &C \left\| \left( \Phi_ { 0 } , \Psi_ { 0 } \right) \right\| ^ { 2 } + C \delta \int _ { 0 } ^ { t } \int_{-\infty}^\infty \Phi_ { \xi \xi } ^ { 2 } d \xi d \tau. \end{aligned}
  \end{equation}
\end{lemma}
\begin{proof}
  Multiply $\eqref{2.3}_1$ and $\eqref{2.3}_2$ by $\Phi$ and $-\frac{\Psi}{p'(\tilde{v})}$ respectively and sum over the results, we have
  \begin{equation}\label{3.6}
    \begin{aligned}
     &\left(\frac{\Phi^2}{2}-\frac { \Psi^ { 2 } } { 2p ^ { \prime } ( \tilde { v } ) }\right)_{t}
      +\left(-\frac{s}{2}\Phi^2-\Phi\Psi-g'(\tilde{v})\Phi\Phi_{\xi}+\frac{s}{2p ' ( \tilde { v } )}\Psi^2\right)_{\xi}\\
      +&g'(\tilde{v})\Phi_{\xi}^2-\frac{s}{2}\left(\frac{1}{p ' ( \tilde { v } )}\right)_{\xi}\Psi^2
      =G\Phi+\frac{p(v|\tilde{v})}{p'(\tilde{v})}\Psi.
    \end{aligned}
  \end{equation}
 It should be noted that there is cancellation in \eqref{3.6} for  $g''(\tilde{v})\tilde{v}_\xi\Phi_\xi\Phi$ which is difficult to control. This is a key observation for the general viscosity $\mu(v)$. Note also that 
  \begin{equation}\label{3.7}
    g'(\tilde{v})=\frac{\mu(\tilde{v})}{\tilde{v}}>0,\qquad
    -\frac{s}{2}\left(\frac{1}{p'(\tilde{v})}\right)_{\xi}=\frac{p''(\tilde{v})}{2\left(p'(\tilde{v})\right)^2}s\tilde{v}_{\xi}\geq0.
  \end{equation}
  Utilize the a priori assumption \eqref{3.1} and Lemma \ref{le3.1}, we obtain that
  \begin{equation}\label{3.8}
    \begin{aligned}
      \left|G\Phi+\frac{p(v|\tilde{v})}{p'(\tilde{v})}\Psi\right|
      &\leq C\left(| \Phi_{ \xi }| | \Phi_{\xi\xi }|+|\tilde { v } _ { \xi } | |  \Phi_ { \xi } | ^ { 2 } \right)|\Phi|+C\|\Psi\|_{L^{\infty}}| \Phi_ { \xi } |^2\\
      &\leq C\|(\Phi,\Psi)\|_{L^{\infty}}(| \Phi_ { \xi } |^2+| \Phi_ { \xi \xi } |^2)\\
      &\leq C\delta(| \Phi_ { \xi } |^2+| \Phi_ { \xi \xi } |^2).
    \end{aligned}
  \end{equation}
  Integrate \eqref{3.6} with respect to $t$ and $\xi$ over $(0, t)\times \mathbb{R}$,  make use of \eqref{3.7}-\eqref{3.8} and choose $\delta$ small enough, we can obtain \eqref{3.5} directly. Thus the proof of Lemma \ref{le3.2} is completed.
\end{proof}

Next we estimate the term $\int _ { 0 } ^ { t } \int_{-\infty}^\infty \Phi_ { \xi \xi } ^ { 2 } d \xi d \tau$.
\begin{lemma}\label{le3.3}
  Under the assumptions of proposition \ref{prop2.1}, it holds that
  \begin{equation}\label{3.9}
    \| ( \Phi, \Psi) ( t ) \| _ { 1 } ^ { 2 } + \int _ { 0 } ^ { t } \left\| \Phi_ { \xi } ( \tau ) \right\| _ { 1 } ^ { 2 } d \tau
    \leq C \left\| \left( \Phi_ { 0 } , \Psi_ { 0 } \right) \right\| _ { 1 } ^ { 2 }
  \end{equation}
\end{lemma}
 \begin{proof}
   Multiply $\eqref{2.3}_1$ and $\eqref{2.3}_2$ by $-\Phi_{\xi\xi}$ and $\frac{\Psi_{\xi\xi}}{p'(\tilde{v})}$ respectively and sum over the results, we have
   \begin{equation}\label{3.10}
    \begin{aligned}
     &\left(\frac{\Phi_{\xi}^2}{2}-\frac { \Psi_{\xi} ^ { 2 } } { 2p ^ { \prime } ( \tilde { v } ) }\right)_{t}
      +\left(\frac{s}{2}\Phi_{\xi}^2-\Phi_t\Phi_{\xi}+\Phi_{\xi}\Psi_{\xi}+\frac{\Psi_t\Psi_{\xi}}{p'(\tilde{v})}-\frac{s}{2p ' ( \tilde { v } )}\Psi_{\xi}^2\right)_{\xi}\\
      &+g'(\tilde{v})\Phi_{\xi\xi}^2-\frac{s}{2}\left(\frac{1}{p ' ( \tilde { v } )}\right)_{\xi}\Psi_{\xi}^2\\
      =&-G\Phi_{\xi\xi}-g''(\tilde { v })\tilde { v }_{\xi}\Phi_{\xi}\Phi_{\xi\xi}-\frac{p(v|\tilde{v})}{p'(\tilde{v})}\Psi_{\xi\xi}
      -\left(\frac{1}{p ' ( \tilde { v } )}\right)_{\xi}(p(v)-p(\tilde{v}))\Psi_{\xi}.
    \end{aligned}
  \end{equation}
  Note that
  \begin{equation}\label{3.11}
    \begin{aligned}
      \left|G\Phi_{\xi\xi}+g''(\tilde { v })\tilde { v }_{\xi}\Phi_{\xi}\Phi_{\xi\xi}\right|
      \leq&C \left( | \Phi_ { \xi }| | \Phi_ { \xi \xi } | + | \tilde { v } _ { \xi }| | \Phi_ { \xi } |^2 \right)|\Phi_{\xi\xi}|+C|\Phi_{\xi}||\Phi_{\xi\xi}| \\
      \leq& (\varepsilon+C\delta)|\Phi_{\xi\xi}|^2+C|\Phi_{\xi}|^2,
    \end{aligned}
  \end{equation}
  \begin{equation}\label{3.12}
    \begin{aligned}
      &\frac{p(v|\tilde{v})}{p'(\tilde{v})}\Psi_{\xi\xi}+\left(\frac{1}{p ' ( \tilde { v } )}\right)_{\xi}(p(v)-p(\tilde{v}))\Psi_{\xi}\\
      =&\left(\frac{p(v|\tilde{v})}{p'(\tilde{v})}\Psi_{\xi}\right)_{\xi}
      +\left(\frac{1}{p'(\tilde { v } )}\right)_{\xi}(p(v)-p(\tilde{v})-p(v|\tilde{v}))\Psi_{\xi}
      -\frac{p(v|\tilde{v})_{\xi}}{p'(\tilde { v })}\Psi_{\xi}\\
      =&\left(\frac{p(v|\tilde{v})}{p'(\tilde{v})}\Psi_{\xi}\right)_{\xi}+\left(\frac{1}{p'(\tilde { v } )}\right)_{\xi}p'(\tilde { v })\Phi_{\xi}\Psi_{\xi}
      -\frac{p(v|\tilde{v})_{\xi}}{p'(\tilde { v })}\Psi_{\xi},
    \end{aligned}
  \end{equation}
and
  \begin{equation}\label{3.13}
    \begin{aligned}
      \left(\frac{1}{p'(\tilde { v } )}\right)_{\xi}p'(\tilde { v })\Phi_{\xi}\Psi_{\xi}
      \leq&-\frac{s}{4}\left(\frac{1}{p ' ( \tilde { v } )}\right)_{\xi}\Psi_{\xi}^2+C|\Phi_{\xi}|^2,\\
      \left|\frac{p(v|\tilde{v})_{\xi}}{p'(\tilde { v })}\Psi_{\xi}\right|
      \leq&C\left(|\Phi_{\xi\xi} | |\Phi_{\xi}| + |\tilde{v}_\xi |\Phi_{\xi}^{2}\right)|\Psi_{\xi}|
      \leq C\delta\left(|\Phi_{\xi\xi} |^2+ |\Phi_{\xi}|^2 \right).
    \end{aligned}
  \end{equation}
Integrate \eqref{3.10} with respect to $t$ and $\xi$ over $(0,t)\times \mathbb{R}$, make use of \eqref{3.11}-\eqref{3.13}, and choose $\varepsilon$ and $\delta$ small enough, we obtain that
  \begin{equation}\label{3.14}
    \begin{aligned}
    & \int_{-\infty}^\infty\left(\Phi_{\xi}^{2}+\Psi_{\xi}^{2}\right)d\xi-s\int_{0}^{t}\int_{-\infty}^\infty\left(\frac{1}{p'(\tilde{v})}\right)_{\xi}\Psi_{\xi}^{2}d\xi d\tau
    +\int_{0}^{t}\int_{-\infty}^\infty \Phi_{\xi\xi}^{2}dxd\tau \\
    \leq &C\left\|\left(\Phi_{0,\xi} , \Psi_{0,\xi} \right) \right\|^{ 2 } + C \int_{0}^{t} \int_{-\infty}^\infty \Phi_{\xi}^{2} d\xi d\tau.
    \end{aligned}
  \end{equation}
  Multiply \eqref{3.5} by a large constant, add \eqref{3.14} and choose  $\delta$ small enough, we can obtain \eqref{3.9} immediately. Lemma \ref{le3.3} is completed.
 \end{proof}

 Since $|\tilde{v}_{\xi}|$ may not have lower bound, we need to estimate the term $\int_{0}^{t}\|\Psi_{\xi}\|^{2} d\tau$.
 \begin{lemma}\label{le3.4}
   Under the assumptions of proposition \ref{prop2.1}, it holds that
   \begin{equation}\label{3.15}
     \int_{0}^{t}\|\Psi_{\xi}\|^{2}d\tau\leq C\left\| \left(\Phi_{0},\Psi_{0}\right)\right\|_{1}^{2}.
   \end{equation}
 \end{lemma}
\begin{proof}
  Multiplying $\eqref{2.3}_1$ by $\Psi_{\xi}$ and make use of $\eqref{2.3}_2$, we have
  \begin{equation}\label{3.16}
    \begin{aligned}
      \Psi_{\xi}^2=& (\Phi\Psi_{\xi})_t+\left[(p(v)-p(\tilde{v}))\Phi-s\Phi\Psi_{\xi}\right]_{\xi}
      -(p(v)-p(\tilde{v}))\Phi_{\xi}-g'(\tilde{v})\Phi_{\xi\xi}\Psi_{\xi}
      -G\Psi_{\xi}.
    \end{aligned}
  \end{equation}
Note that
  \begin{equation}\label{3.17}
    \begin{aligned}
      &\left|(p(v)-p(\tilde{v}))\Phi_{\xi}+g'(\tilde{v})\Phi_{\xi\xi}\Psi_{\xi}+G\Psi_{\xi}\right|\\
      \leq& \frac{1}{2}\Psi_{\xi}^2+C(|\Phi_{\xi}|^2+|\Phi_{\xi\xi}|^2+|G|^2)
      \leq\frac{1}{2}\Psi_{\xi}^2+C(|\Phi_{\xi}|^2+|\Phi_{\xi\xi}|^2).
    \end{aligned}
  \end{equation}
Integrate \eqref{3.16} with respect to $t$ and $\xi$ over $(0,t)\times \mathbb{R}$ and with the help of \eqref{3.17}, we have
  \begin{equation}\label{3.18}
    \begin{aligned}
      \int_{0}^{t}\|\Psi_{\xi}\|^{2}d\tau\leq C\|(\Phi,\Psi_{\xi})\|^2+C\|(\Phi_0,\Psi_{0,\xi})\|^2+\int_{0}^{t}\|\Phi_{\xi}\|_1^{2}d\tau,
    \end{aligned}
  \end{equation}
  where we have used the fact that
  \begin{equation*}
    \int_{0}^{t}\int_{-\infty}^\infty(\Phi\Psi_{\xi})_\tau d\xi d\tau=\int(\Phi\Psi_{\xi}-\Phi_0\Psi_{0,\xi})d\xi\leq \|(\Phi,\Psi_{\xi})\|^2+\|(\Phi_0,\Psi_{0,\xi})\|^2.
  \end{equation*}
Lemma \ref{le3.4} is completed from Lemma \ref{le3.3} and \eqref{3.18}.
\end{proof}

\begin{lemma}\label{le3.5}
  Under the assumptions of proposition \ref{prop2.1}, it holds that
   \begin{equation}\label{3.19}
     \left\| \left( \Phi_ { \xi \xi } , \Psi_ { \xi \xi } \right) ( t ) \right\| ^ { 2 }
     + \int _ { 0 } ^ { t } \int \Phi_ { \xi \xi \xi } ^ { 2 } d \xi d \tau
     \leq C \left\| \left( \Phi_ { 0 } , \Psi_ { 0 } \right) \right\| _ { 2 } ^ { 2 }.
   \end{equation}
\end{lemma}
\begin{proof}
Firstly, differentiating $\eqref{2.3}_1$ with respect to $\xi$ twice, then multiplying the result by $\Phi_{\xi\xi}$; secondly, multiplying $\eqref{2.3}_2$ by $-\frac { 1 } { p'( \tilde { v } ) }$, differentiating the result with respect to $\xi$ twice and multiplying the result by $\Psi_{\xi\xi}$; at last we sum the two results together to obtain that
 \begin{equation}\label{3.20}
   \begin{aligned}
     &\left(\frac{\Phi_{\xi\xi}^2}{2}-\frac { \Psi_{\xi\xi} ^ { 2 } } { 2p ^ { \prime } ( \tilde { v } ) }\right)_{t}
     +g'(\tilde{v})\Phi_{\xi\xi\xi}^2-\frac{s}{2}\left(\frac{1}{p ' ( \tilde { v } )}\right)_{\xi}\Psi_{\xi\xi}^2\\
      &+\left(-\frac{s}{2}\Phi_{\xi\xi}^2-(g'(\tilde{v})\Phi_{\xi\xi})_{\xi}\Phi_{\xi}-\Phi_{\xi\xi}\Psi_{\xi\xi}+\frac{s}{2p'( \tilde { v } )}\Psi_{\xi\xi}^2\right)_{\xi}\\
      =&G_{\xi\xi}\Phi_{\xi\xi}+g'(\tilde { v })_{\xi\xi\xi}\Phi_{\xi}\Phi_{\xi\xi}+2g'(\tilde { v })_{\xi\xi}\Phi_{\xi\xi}^2
      +\frac{p(v|\tilde{v})_{\xi\xi}}{p'(\tilde{v})}\Psi_{\xi\xi}\\
      &-2\left(\frac{1}{p ' ( \tilde { v } )}\right)_{\xi}(p'(\tilde{v})\Phi_{\xi})_{\xi}\Psi_{\xi\xi}
      -\left(\frac{1}{p ' ( \tilde { v } )}\right)_{\xi\xi}p'(\tilde{v})\Phi_{\xi}\Psi_{\xi\xi}.
    \end{aligned}
 \end{equation}
 Integrate \eqref{3.20} with respect to $t$ and $\xi$ over $(0,t)\times \mathbb{R}$, we have
 \begin{equation}\label{3.21}
   \begin{aligned}
     &\int \left(\frac{\Phi_{\xi\xi}^2}{2}-\frac { \Psi_{\xi\xi} ^ { 2 } } { 2p ^ { \prime } ( \tilde { v } ) }\right)d\xi
   +\int_{0}^{t}\int \left[g'(\tilde{v})\Phi_{\xi\xi\xi}^2-\frac{s}{2}\left(\frac{1}{p ' ( \tilde { v } )}\right)_{\xi}\Psi_{\xi\xi}^2\right]d\xi d\tau\\
   =&\int \left(\frac{\Phi_{0,\xi\xi}^2}{2}-\frac { \Psi_{0,\xi\xi} ^ { 2 } } { 2p ^ { \prime } ( \tilde { v } ) }\right)d\xi
   +\underbrace{\int_{0}^{t}\int G_{\xi\xi}\Phi_{\xi\xi}d\xi d\tau}_{I_1}\\
   &+\underbrace{\int_{0}^{t}\int \left(g'(\tilde { v })_{\xi\xi\xi}\Phi_{\xi}\Phi_{\xi\xi}+2g'(\tilde { v })_{\xi\xi}\Phi_{\xi\xi}^2\right)d\xi d\tau}_{I_2}\\
   &+\underbrace{\int_{0}^{t}\int \frac{p(v|\tilde{v})_{\xi\xi}}{p'(\tilde{v})}\Psi_{\xi\xi}d\xi d\tau}_{I_3}
   \underbrace{-2\int_{0}^{t}\int \left(\frac{1}{p ' ( \tilde { v } )}\right)_{\xi}(p'(\tilde{v})\Phi_{\xi})_{\xi}\Psi_{\xi\xi}d\xi d\tau}_{I_4}\\
   &\underbrace{-\int_{0}^{t}\int \left(\frac{1}{p ' ( \tilde { v } )}\right)_{\xi\xi}p'(\tilde{v})\Phi_{\xi}\Psi_{\xi\xi}d\xi d\tau}_{I_5}.
   \end{aligned}
 \end{equation}

By making use of Lemma \ref{le3.1}, Sobolev inequality and the a priori assumption \eqref{3.2}, we have
 \begin{eqnarray*}
   |I_1|&=&\left|\int_{0}^{t}\int G_{\xi}\Phi_{\xi\xi\xi}d\xi d\tau\right|\\
     &\leq&C\int_{0}^{t}\int\left(\Phi_{\xi\xi}^{2}+|\Phi_{\xi}| |\Phi_{\xi\xi}|^2+|\Phi_{\xi}||\Phi_{\xi\xi\xi}|+\left|\Phi_{\xi}\right|^2\right)|\Phi_{\xi\xi\xi}|d\xi d\tau\\
     &\leq&C\int_{0}^{t}\|\Phi_{\xi\xi}\|_{L^{\infty}}\|\Phi_{\xi\xi}\|\|\Phi_{\xi\xi\xi}\|d\tau+C\delta\int_{0}^{t}(\|\Phi_{\xi}\|^2+\|\Phi_{\xi\xi\xi}\|^2)d\tau\\
     &\leq&C\int_{0}^{t}\|\Phi_{\xi\xi}\|^{\frac{3}{2}}\|\Phi_{\xi\xi\xi}\|^{\frac{3}{2}}d\tau+C\delta\int_{0}^{t}(\|\Phi_{\xi}\|^2+\|\Phi_{\xi\xi\xi}\|^2)d\tau\\
     %&\leq&C\delta\int_{0}^{t}\|\Phi_{\xi\xi}\|^{\frac{1}{2}}\|\Phi_{\xi\xi\xi}\|^{\frac{3}{2}}d\tau+C\delta\int_{0}^{t}(\|\Phi_{\xi}\|^2+\|\Phi_{\xi\xi\xi}\|^2)d\tau\\
     &\leq&C\delta\int_{0}^{t}(\|\Phi_{\xi}\|^2+\|\Phi_{\xi\xi}\|^2+\|\Phi_{\xi\xi\xi}\|^2)d\tau,
 \end{eqnarray*}
 and 
  \begin{equation*}
  \begin{aligned}
  |I_3|\leq&C\int_{0}^{t}\int\left(|\Phi_{\xi\xi\xi}| |\Phi_{\xi}|+\Phi_{\xi}^{2}+\Phi_{\xi\xi}^{2}+|\Phi_{\xi}||\Phi_{\xi\xi}|^2 \right)|\Psi_{\xi\xi}|d\xi d\tau\\
  \leq&C\int_{0}^{t}\|\Phi_{\xi}\|_{L^{\infty}}(\|\Phi_{\xi\xi\xi}\|+\|\Phi_{\xi}\|)\|\Psi_{\xi\xi}\|d\tau
  +C\int_{0}^{t}\|\Phi_{\xi\xi}\|_{L^{\infty}}\|\Phi_{\xi\xi}\|\|\Psi_{\xi\xi}\|d\tau\\
  \leq&C\delta\int_{0}^{t}(\|\Phi_{\xi}\|+\|\Phi_{\xi\xi}\|)(\|\Phi_{\xi\xi\xi}\|+\|\Phi_{\xi}\|)d\tau
  +C\delta\int_{0}^{t}(\|\Phi_{\xi\xi}\|+\|\Phi_{\xi\xi\xi}\|)\|\Phi_{\xi\xi}\|d\tau\\
  \leq&C\delta\int_{0}^{t}(\|\Phi_{\xi}\|^2+\|\Phi_{\xi\xi}\|^2+\|\Phi_{\xi\xi\xi}\|^2)d\tau.
  \end{aligned}
  \end{equation*}
On the other hand, it is straightforward to imply that
 \begin{equation*}
   |I_2|\leq C\int_{0}^{t}(\|\Phi_{\xi}\|^2+\|\Phi_{\xi\xi}\|^2)d\tau,
 \end{equation*}
 \begin{equation*}
   \begin{aligned}
     |I_4|\leq&-\frac{s}{4}\int_{0}^{t}\int\left(\frac{1}{p ' ( \tilde { v } )}\right)_{\xi}\Psi_{\xi\xi}^2d\xi d\tau
     +C\int_{0}^{t}(\|\Phi_{\xi}\|^2+\|\Phi_{\xi\xi}\|^2)d\tau.
   \end{aligned}
 \end{equation*}
and
 \begin{equation*}
   \begin{aligned}
     I_5=&\int_{0}^{t}\int\left[\left(\frac{1}{p ' ( \tilde { v } )}\right)_{\xi\xi}p'(\tilde{v})\Phi_{\xi}\right]_{\xi}\Psi_{\xi}d\xi d\tau\\
     \leq&C\int_{0}^{t}(\|\Phi_{\xi}\|^2+\|\Phi_{\xi\xi}\|^2+\|\Psi_{\xi}\|^2)d\tau.
   \end{aligned}
 \end{equation*}
Plug the estimates of $I_1-I_5$ into \eqref{3.21}, choose $\delta$ small enough and make use of the Lemma \ref{le3.3}-\ref{le3.4}, we can obtain  \eqref{3.19}. Thus  the proof of Lemma \ref{le3.5} is completed.
\end{proof}

Finally we estimate $\int_{0}^{t}\|\Psi_{\xi\xi}\|^2d\tau$.
\begin{lemma}\label{le3.6}
  Under the assumptions of proposition \ref{prop2.1}, it holds that
  \begin{equation}\label{3.22}
    \int_{0}^{t}\left\|\Psi_{\xi\xi}(\tau)\right\|^{2}d\tau \leq C\left\| \left( \Phi_{0},\Psi_{0} \right) \right\|_{2}^{2}.
  \end{equation}
\end{lemma}
\begin{proof}
   Multiplying $\eqref{2.3}_1$ by $\Psi_{\xi\xi\xi}$ and using  $\eqref{2.3}_2$, we have
   \begin{equation}\label{3.23}
     \begin{aligned}
       \Psi_{\xi\xi}^2=&(\Phi_{\xi}\Psi_{\xi\xi})_t+\left(-\Phi_t\Psi_{\xi\xi}+\Psi_{\xi}\Psi_{\xi\xi}+(g'(\tilde{v})\Phi_{\xi})_{\xi}\Psi_{\xi\xi}\right)_{\xi}\\
       &-(p(v)-p(\tilde{v}))_{\xi\xi}\Phi_{\xi}-(g'(\tilde{v})\Phi_{\xi})_{\xi\xi}\Psi_{\xi\xi}+G\Psi_{\xi\xi\xi}.
     \end{aligned}
   \end{equation}
Integrate the above equation with respect to $t$ and $\xi$ over $(0,t)\times \mathbb{R}$, we obtain that 
   \begin{equation}\label{3.24}
     \begin{aligned}
       \int_{0}^{t}\left\|\Psi_{\xi\xi}(\tau)\right\|^{2}d\tau=&\underbrace{\int\left(\Phi_{\xi}\Psi_{\xi\xi}-\Phi_{0,\xi}\Psi_{0,\xi\xi}\right)d\xi}_{I_6}
     \underbrace{-\int_{0}^{t}\int (p(v)-p(\tilde{v}))_{\xi\xi}\Phi_{\xi}  d\xi d\tau}_{I_7}\\
     &\underbrace{-\int_{0}^{t}\int (g'(\tilde{v})\Phi_{\xi})_{\xi\xi}\Psi_{\xi\xi}  d\xi d\tau}_{I_8}
     +\underbrace{\int_{0}^{t}\int G\Psi_{\xi\xi\xi} d\xi d\tau}_{I_9}.
     \end{aligned}
   \end{equation}
It is straightforward to imply that
   \begin{equation*}
     \begin{aligned}
       |I_6|\leq &\|(\Phi_{\xi},\Psi_{\xi\xi})\|^2+\|(\Phi_{0,\xi},\Psi_{0,\xi\xi})\|^2,\\
       |I_8|\leq&\frac{1}{4}\int_{0}^{t}\left\|\Psi_{\xi\xi}(\tau)\right\|^{2}d\tau+C\int_{0}^{t}\|(\Phi_{\xi},\Phi_{\xi\xi},\Phi_{\xi\xi\xi})\|^2d\tau,
     \end{aligned}
   \end{equation*}
   and 
   \begin{equation*}
     \begin{aligned}
       |I_7|=&\left|\int_{0}^{t}\int (p(v)-p(\tilde{v}))\Phi_{\xi\xi\xi}  d\xi d\tau\right|\leq C\int_{0}^{t}\|(\Phi_{\xi},\Phi_{\xi\xi\xi})\|^2d\tau,\\
       |I_9|=&\left|\int_{0}^{t}\int G_\xi \Psi_{\xi\xi} d\xi d\tau\right|\\
       \leq&C\int_{0}^{t}\int\left(\Phi_{\xi\xi}^{2}+|\Phi_{\xi}| |\Phi_{\xi\xi}|^2+|\Phi_{\xi}| |\Phi_{\xi\xi\xi}|+\left|\Phi_{\xi}\right|^2\right)|\Psi_{\xi\xi}|d\xi d\tau\\
       \leq&C\int_{0}^{t}\|\Phi_{\xi\xi}\|_{L^{\infty}}\|\Phi_{\xi\xi}\|\|\Psi_{\xi\xi}\|d\tau
       +C\int_{0}^{t}\|\Phi_{\xi}\|_{L^{\infty}}(\|\Phi_{\xi}\|+\|\Phi_{\xi\xi\xi}\|)\|\Psi_{\xi\xi}\|d\tau\\
       \leq&C\delta\int_{0}^{t}(\|\Phi_{\xi\xi}\|+\|\Phi_{\xi\xi\xi}\|)\|\Phi_{\xi\xi}\|+(\|\Phi_{\xi}\|+\|\Phi_{\xi\xi}\|)(\|\Phi_{\xi}\|+\|\Phi_{\xi\xi\xi}\|)d\tau\\
       \leq&C\delta\int_{0}^{t}\|(\Phi_{\xi},\Phi_{\xi\xi},\Phi_{\xi\xi\xi})\|^2d\tau,
     \end{aligned}
   \end{equation*}
   plug the estimates of $I_6-I_9$ into \eqref{3.24}, combine with Lemma \ref{le3.5}, we can show \eqref{3.22}. Then the Lemma \ref{le3.6} is completed.
\end{proof}
Similarly we can get the higher order estimates and the proof is omitted for brevity.   
\begin{lemma}\label{le3.7}
	Under the assumptions of proposition \ref{prop2.1}, it holds that
	\begin{equation}\label{4.25}
	\left\| \left( \Phi_ { \xi \xi \xi} , \Psi_ { \xi \xi \xi} \right) ( t ) \right\| ^ { 2 }
	+ \int _ { 0 } ^ { t } \int [\Phi_ { \xi \xi \xi\xi } ^ { 2 } +\Psi_{\xi\xi\xi}^2 ]d \xi d \tau
	\leq C \left\| \left( \Phi_ { 0 } , \Psi_ { 0 } \right) \right\| _ { 3 } ^ { 2 }.
	\end{equation}
\end{lemma}

\begin{proof}[Proof of proposition \ref{prop2.1}]
Proposition  \ref{prop2.1} is directly proved from Lemma \ref{le3.3}-\ref{le3.7}.
\end{proof}

 \section{Proof of Theorem \ref{th1}}
 Now we turn to prove our main theorem. It is straightforward to imply  \eqref{1.18} from Theorem \ref{th3}.   We will follow the ideas of \cite{Matsumura-Nishihara-1985-JJAM} to show \eqref{1.19} with the help of the following useful lemma. 
 
\begin{lemma}[\cite{Matsumura-Nishihara-1985-JJAM}]\label{asymptotic}
Assume that the function $f(t)\ge 0\in L^1(0,+\infty)\cap BV(0,+\infty)$, then it holds that $f(t)\to 0$ as $t\to +\infty$.
\end{lemma}

Let's turn to the system \eqref{2.3}. Differentiating $\eqref{2.3}_1$ twice with respect to $\xi$, multiplying the resulting equation by $\Phi_{\xi\xi}$ and integrating it with respect to $\xi$ on $(-\infty,\infty)$, we have 
 \begin{equation}\label{5.1}
\left|\frac{d}{dt}(\|\Phi_{\xi\xi}\|^2)\right|
 \le C(\|\Phi_\xi\|^2_{2}+\|\Psi_{\xi\xi}\|^2)+|\int_{-\infty}^\infty G_\xi\Phi_{\xi\xi\xi}d\xi|\le C(\|\Phi_\xi\|^2_{2}+\|\Psi_{\xi\xi}\|^2),
 \end{equation}
which indicates from proposition \ref{prop2.1} that 
\begin{equation}\label{5.2}
\int_0^\infty \left|\frac{d}{dt}(\|\Phi_{\xi\xi}\|^2)\right|d\tau
\le C.
\end{equation}
This, together with proposition \ref{prop2.1}, gives  that $\|\Phi_{\xi\xi}\|^2\in L^1(0,+\infty)\cap BV(0,+\infty)$ and further implies that 
\begin{equation}\label{5.3}
\|\Phi_{\xi\xi}\|\to 0~ \mbox{as}~ t\to +\infty.
\end{equation}
The Sobolev inequality implies that 
\begin{equation}\label{5.4}
\begin{aligned}
&\|v-\tilde{v}\|_{L^\infty}(t)=\|\Phi_\xi\|_{L^\infty}(t)\le \sqrt{2}\|\Phi_\xi\|^\frac12\|\Phi_{\xi\xi}\|^\frac12\le C\|\Phi_{\xi\xi}\|^\frac12\to 0~ \mbox{as}~ t\to +\infty,\\
&\|v_\xi-\tilde{v}_\xi\|_{L^\infty}(t)=\|\Phi_{\xi\xi}\|_{L^\infty}(t)\le \sqrt{2}\|\Phi_{\xi\xi}\|^\frac12\|\Phi_{\xi\xi\xi}\|^\frac12\le C\|\Phi_{\xi\xi}\|^\frac12\to 0~ \mbox{as}~ t\to +\infty.
\end{aligned}
\end{equation}

Similarly we can obtain that 
\begin{equation}\label{5.5}
\int_0^\infty \left|\frac{d}{dt}(\|\Psi_{\xi\xi}\|^2)\right|d\tau
\le C,
\end{equation}
which gives that $\|\Psi_{\xi\xi}\|(t)\to 0$ as $t\to \infty$ and 
\begin{equation}\label{5.6}
\|h-\tilde{h}\|_{L^\infty}(t)\le C\|\Psi_{\xi\xi}\|^{\frac12}\to 0,~\mbox{as}~t\to \infty.
\end{equation}
Note that  \begin{equation}\label{5.7}
u-\tilde{u}=h-\tilde{h}+g(v)_\xi-g(\tilde{v})_\xi,
\end{equation} 
we obtain from \eqref{5.4} and \eqref{5.7} that 
\begin{equation}\label{5.8}
\|u-\tilde{u}\|_{L^\infty}(t)\le \|h-\tilde{h}\|_{L^\infty}(t)+C\left(\|v-\tilde{v}\|_{L^\infty}(t)+\|v_\xi-\tilde{v}_\xi\|_{L^\infty}(t)\right)\to 0, ~\mbox{as}~t\to \infty.
\end{equation}
Therefore the proof of Theorem \ref{th1} is completed.

\

\smallskip
\section*{Acknowledgments}
Lin He is partially supported by the Fundamental Research Funds for the Central Universities. Feimin Huang is partially supported by National Center for Mathematics and Interdisciplinary Sciences, AMSS, CAS and NSFC Grant No. 11371349 and 11688101.

\end{document}